\documentclass[12pt]{amsart}
\usepackage{rr}
\usepackage{hyperref}
\usepackage{mathtools}

\begin{document}

\title{Goldbach Theorems for Group Semidomains}
\author{Eddy Li, Advaith Mopuri, and Charles Zhang}
\date{17 August 2024}
\begin{abstract}
A semidomain is a subsemiring of an integral domain. We call a semidomain $S$ additively reduced if $0$ is the only invertible element of the monoid $(S, +)$, while we say that $S$ is additively Furstenberg if every non-invertible element of $(S,+)$ can be expressed as the sum of an atom and an element of $S$. In this paper, we study a variant of the Goldbach conjecture within the framework of group semidomains \( S[G] \) and group series semidomains \( S\sg \), where \( S \) is both an additively reduced and additively Furstenberg semidomain and \( G \) is a torsion-free abelian group. In particular, we show that every non-constant polynomial expression in $S[G]$ can be written as the sum of at most two irreducibles if and only if the condition \(\mathscr{A}_+(S) = S^\times\) holds.
\end{abstract}
\maketitle

\section{Introduction}

\smallskip

The Goldbach conjecture is a long-standing problem in number theory that posits that every even integer greater than $2$ can be expressed as the sum of two prime numbers. First introduced in a letter between Goldbach and Euler in 1742, this conjecture has been extensively tested for large integers, yet a complete proof has eluded mathematicians. The conjecture's simplicity in statement, combined with its profound complexity, has inspired diverse research avenues and numerous analogues.

Many variants of the Goldbach conjecture have been formulated for classes of polynomial rings. For instance, Hayes~\cite{hayes} proved that every non-constant polynomial $f \in \ZZ[x]$ can be written as the sum of two irreducible polynomials of degree at most $\deg(f)$. This result was subsequently improved by Kozek~\cite{kozek} and Saidak~\cite{saidak}, and later extended by Pollack~\cite{pollack}. Meanwhile, analogues of the Goldbach conjecture for polynomials with coefficients in a finite field have been studied by Bender~\cite{bender}, Car and Gallardo~\cite{car}, and Effinger and Hayes~\cite{effinger}. More recently, Paran~\cite{paran} investigated a version of the Goldbach conjecture in the ring of formal power series over the integers. 

Although some previous formulations of Goldbach-type theorems have been set within the context of commutative rings, we suggest that semirings lacking nontrivial additive inverses offer a more natural algebraic framework for such analogues, as this aligns with Goldbach's focus on positive elements. Inspired by the previous observation, Liao and Polo~\cite{liaopolo} initiated the study of Goldbach analogues in the context of polynomial semirings. Specifically, they showed that every non-constant (Laurent) polynomial in $\NN_0[x^{\pm 1}]$ can be written as the sum of at most two irreducibles (\!\!\cite[Theorem~1.1]{liaopolo}). This result closely aligns with the spirit of the Goldbach conjecture. Specifically, since the units in $\mathbb{N}_0[x^{\pm 1}]$ are precisely the integer powers of $x$, proving \cite[Theorem~1.1]{liaopolo} involves partitioning a given set of units into two subsets, each corresponding to an irreducible polynomial with positive integer coefficients.

A semidomain is a special type of semiring contained in an integral domain. One can think of semidomains as integral domains that do not require additive inverses. For example, both $\NN_0$ and $\NN_0[x]$ are semidomains. The algebraic and factorization properties of semidomains have generated some interest lately (see, for instance, \cite{CF19, chapmanpolo, fox2023, polo}). We call a semidomain $S$ additively reduced if the monoid $(S, +)$ contains exactly one invertible element (i.e., $0$). Kaplan and Polo~\cite{kaplan2023goldbach} extended the study of Goldbach analogues to the class of polynomials with coefficients in an additively reduced semidomain. In particular, they provided a simpler proof of \cite[Theorem~1.1]{liaopolo} not relying on assumptions about the distribution of primes in $\mathbb{Z}$, some of which played a crucial role in \cite{liaopolo}. In this paper, we study a variant of the Goldbach conjecture within the framework of group semidomains \( S[G] \) and group series semidomains \( S\sg \), where \( S \) is an additively reduced semidomain and \( G \) is a torsion-free abelian group.

In Section $2$, we review some of the standard notation and terminology we used throughout.

For an additively reduced semidomain $S$ and a torsion-free group $G$, we define the group semidomain \( S[G] \) to be the semidomain of all formal finite sums with coefficients in $S$ and exponents in $G$, where addition and multiplication are defined as for polynomials. In Section~3, we describe the group semidomains $S[G]$ with $S$ additively Furstenberg (i.e., every nonzero element of $S$ can be expressed as the sum of an atom and an element of $S$) that satisfies an analogue of the Goldbach conjecture. 

We conclude in Section~4, by providing analogues of the weak Goldbach conjecture for $S\sg$, the semidomain of infinite formal sums with coefficients in $S$ and exponents in $G$, when $S$ is additively Furstenberg. In this context, we characterize the group series semidomains $S\sg$, where $G$ is finitely generated, that satisfy a variant of the weak Goldbach conjecture.

\section{Background}

\smallskip

Let us review some of the standard notation and terminology that we will use later. For a comprehensive background on semiring theory, we recommend the monograph~\cite{golan}. Let $\mathbb{Z}$, $\mathbb{N}$, and $\mathbb{N}_0$ denote the sets of integers, positive integers, and nonnegative integers, respectively. For $m, n \in \mathbb{N}_0$, we define $\llbracket m, n \rrbracket$ as the set $\{k \in \mathbb{N}_0 \mid m \leq k \leq n\}$.

Throughout this paper, a \emph{monoid} is a semigroup with an identity element denoted by $0$. We shall tacitly assume that a monoid is always commutative and cancellative and, unless we specify otherwise, we always use additive notation for monoids. For the remainder of this section, let $M$ be a monoid. The set of invertible elements of $M$ is denoted by $\unit(M)$. We say that $M$ is \emph{reduced} if $\unit(M) = \{0\}$. For $b,c \in M$, we say that $b$ \emph{divides} $c$ in $M$, denoted $b\mid_M c$, if there exists $b' \in M$ such that $c = b + b'$. If the monoid $M$ is clear from context, we simply write $b\mid c$. A submonoid $N$ of $M$ is \emph{divisor-closed} if for every $b \in N$ and $c \in M$ the relation $c \mid_M b$ implies that $c \in N$. Let $S$ be a nonempty subset of $M$. We use the term \emph{common divisor of $S$} to refer to an element $d\in M$ that divides all elements of $S$. We call a common divisor $d$ of $S$ a \emph{greatest common divisor} if it is divisible by every common divisor of $S$. We denote by $\gcdd_M(S)$ the set consisting of all greatest common divisors of $S$ and drop the subscript when there is no risk of confusion.

An element $a \in M \setminus \unit(M)$ is called an \emph{atom} (or \emph{irreducible}) if for every $b, c \in M$ the equality $a = b + c$ implies that either $b \in \unit(M)$ or $c \in \unit(M)$. We denote by $\atom(M)$ the set of atoms of $M$. The monoid $M$ is \emph{atomic} provided that every $m \in M \setminus \unit(M)$ can be written as a finite sum of atoms. On the other hand, $M$ is said to be \emph{Furstenberg} if every $m\in M\setminus \unit(M)$ is divisible by an atom of $M$. Observe that every atomic monoid is Furstenberg. 

If every element in a monoid $M$ has an inverse, then we say that $M$ is an \emph{abelian group}. Given our assumption of commutativity for monoids, all groups mentioned will be assumed to be abelian. A group $G$ is called \emph{torsion-free} if for $g_1, g_2\in G$ and $n\in \NN$, the equality $ng_1=ng_2$ of addition repeated $n$ times implies that $g_1=g_2$. As established by Levi~\cite{levi1913}, every (abelian) group that is torsion-free also has a linear order $(\le)$ that is compatible with the group operation. In other words, the relation $(\le)$ imposes an ordering on $G$ for which any two group elements are comparable, with the additional requirement that $b + d\le c + d$ holds if and only if $b\le c$ for all $b,c,d\in G$.

A \emph{semiring} \( S \) is a nonempty set with two binary operations, \emph{addition} \((+)\) and \emph{multiplication} \((\cdot)\). The structure \((S, +)\) forms a monoid with identity element \( 0 \), and \((S \setminus \{0\}, \cdot)\) is a non-cancellative monoid with identity element \( 1 \). Additionally, multiplication distributes over addition, so that \( b(c + d) = bc + bd \) for all \( b, c, d \in S \). In other contexts, commutativity of \((S \setminus \{0\}, \cdot)\) may not be necessary, but it will be required in this paper. A subset $S'$ of a semiring $S$ is a \emph{subsemiring} if $(S', +)$ is a submonoid of $(S, +)$ that contains $1$ and is closed under multiplication. Note that every subsemiring of $S$ is also a semiring. We now define the most frequently used structure in this paper: semidomains.

\begin{defn}
    A \emph{semidomain} is a subsemiring of an integral domain.
\end{defn}

For the rest of this section, let $S$ be a semidomain. We say that $(S \setminus \{0\}, \cdot)$ is the \emph{multiplicative monoid} of $S$, denoted by $S^*$. Following standard notation from ring theory, we refer to the invertible elements of the monoid $S^*$ as \emph{units}, and we denote the set (which is, in fact, a group) of units of $S$ by $S^\times$. For $b, c \in S^*$ such that $b$ divides $c$ in $S^*$, we write $b \mid_S c$ (instead of $b \mid_{S^*} c$). Also, for a nonempty subset $B$ of $S$, the term $\gcdd(B)$ represents the set of greatest common divisors of $B$ in the monoid $S^*$. Additionally, we denote the set of atoms of the multiplicative monoid $S^*$ by $\atom(S)$ instead of $\atom(S^*)$, and we denote the set of atoms of the additive monoid $(S, +)$ by $\atom_{+}(S)$. Finally, a semidomain $S$ is \emph{additively reduced} (resp., \emph{additively Furstenberg}) if the monoid $(S,+)$ is reduced (resp., Furstenberg). 

Let \( S \) be a semidomain and let \( G \) be a torsion-free (abelian) group. The \emph{group semidomain} \( S[G] \) is the set of all polynomial expressions
\[
S[G] = \left\{ \sum_{g \in G} s_g x^g \, \middle| \, s_g \in S \text{ and } s_g = 0 \text{ for all but finitely many } g \in G \right\}.
\]
The addition and multiplication operations in \( S[G] \) are defined as for polynomials: For elements \( f = \sum_{g \in G} s_g x^g \) and \( h = \sum_{g \in G} s_g' x^g \) in \( S[G] \), their sum is given by
   \[
   f + h = \sum_{g \in G} (s_g + s_g') x^g,
   \]
where the sum \( s_g + s_g' \) is taken in \( S \). The product of \( f = \sum_{g \in G} s_g x^g \) and \( h = \sum_{r \in G} s_r x^r \) is defined by
\[
   f \cdot h = \sum_{k \in G} \left( \sum_{\substack{g, r \in G \\ g + r = k}} s_g s_r \right) x^k,
\]
where the inner product \( s_g s_r \) is taken in \( S \) and \( g + r \) is the group operation on \( G \). It can be verified that \( S[G] \), with these operations, forms a semidomain. From this point on, unless specified otherwise, we assume that in any polynomial expression \( f = \sum_{i=0}^n s_i x^{g_i} \), the exponents satisfy \( g_0 > g_1 > \dots > g_n \) and each \( s_i \) belongs to \( S^* \) for every \( i \in \mathbb{N}_0 \). We define the \emph{degree} of $f$ as $\deg(f)=\max\{g_0,g_1,\dots,g_n\}=g_0$. In a similar spirit, we can define the \emph{group series semidomain} $S\sg$ as the set of all series expressions
\[  
    S\sg = \left\{ \sum_{i = 0}^{\infty} s_i x^{g_i} \, \middle| \, s_i \in S,\, g_i \in G, \text{ and } g_i < g_{i + 1} \text{ for every } i \in \NN_0 \right\}.
\]
The operations of addition and multiplication in $S\sg$ are defined analogously to those in $S[G]$. From this point on, unless specified otherwise, we assume that in any series expression $f = \sum_{i = 0}^{\infty} s_ix^{g_i}$, the coefficients satisfy that if $s_i = 0$ then $s_{i + 1} = 0$ for every $i \in \NN_0$. Observe that if $S$ is additively reduced, then $S[G]$ is a divisor-closed submonoid of $S\llbracket G \rrbracket$. Given an element $f \in S\sg$, we denote by $\supp(f)=\{g_i\mid s_i\neq0\}$ the set of exponents in the canonical representation of $f$. We refer to this set as the \emph{support of $f$}. For an additively reduced semidomain $S$ and a torsion-free (abelian) group $G$, it is not hard to check that
$$S\llbracket G\rrbracket^\times = S[G]^\times = \{sx^g\mid s\in S^{\times} \text{ and } g\in G\}.$$ Finally, an element $f\in S\llbracket G\rrbracket^*$ is called \emph{monolithic} if for $p,q\in S\llbracket G \rrbracket$, the condition $f=pq$ implies that either $p$ or $q$ is a monomial. Monolithicity represents a relaxed version of the concept of irreducibility.

\section{A Goldbach Theorem for Group Semidomains}
\smallskip
In this section, we establish an analogue of the Goldbach conjecture for group semidomains $S[G]$, where $S$ is an additively reduced and additively Furstenberg semidomain and $G$ is a torsion-free group. More specifically, we show that every non-constant polynomial expression in $S[G]$ can be written as the sum of at most two irreducibles if and only if $\atom_+(S)=S^\times$, which extends~\cite[Theorem~3.6]{kaplan2023goldbach}. Our approach closely follows key ideas presented in~\cite{kaplan2023goldbach}.

\begin{lemma}\label{lem31}
    Let \( S \) be an additively reduced semidomain and \( G \) a torsion-free (abelian) group. Then $f=\sum_{i=0}^n s_ix^{g_i} \in S[G]$ with $|\supp(f)|>1$ is irreducible if and only if $f$ is monolithic and $1\in\gcdd(\{s_0,s_1,\dots,s_n\})$.
\end{lemma}
\begin{proof}
Suppose that $f$ is irreducible. Now if $f = pq$ for some $p, q \in S[G]^*$, then either $p$ or $q$ is a (multiplicative) unit and thus a monomial. Hence $f$ is monolithic. Now, take $t \in S^*$ to be a common divisor of $s_0, \ldots, s_n$. Since $f$ is irreducible, we obtain that $t \in S^{\times}$, which concludes the proof of the direct implication. Conversely, suppose that $f$ is monolithic and $1\in\gcdd(\{s_0,s_1,\dots,s_n\})$. Assume we can write $f = pq$ for some $p, q \in S[G]^*$. Since $f$ is monolithic, we may assume that $p = sx^g$ for some $s \in S$ and $g \in G$. Note that $s \in S^{\times}$ because $1\in\gcdd(\{s_0,s_1,\dots,s_n\})$. Consequently, the factor $p$ is necessarily a unit which, in turn, implies that $f$ is irreducible.  
\end{proof}

The following two lemmas establish criteria for when polynomial expressions in a group semidomain are monolithic.
\begin{lemma}\label{lem32}
    Let \( S \) be an additively reduced semidomain and \( G \) a torsion-free (abelian) group. For an element $f=\sum_{i=0}^ns_ix^{g_i}\in S[G]$, the following statements hold.
    \smallskip
    \begin{itemize}
        \item[(1)] If $|\supp(f)|\ge2$ and $2g_1<g_0+g_n$, then $f$ is monolithic.
        \item[(2)] If $|\supp(f)|>3$ and $2g_1\le g_0+g_n$, then $f$ is monolithic.
    \end{itemize}
\end{lemma}
\begin{proof}
Scaling by an appropriate power of $x$, we may assume that $g_n = 0$. Suppose that $f$ is not monolithic. Then we can write $f = pq$, where $p$ and $q$ are not monomial expressions in $S[G]$. Again, scaling by appropriate powers of $x$, we may assume that neither $p(0)$ nor $q(0)$ equals $0$. Since $S$ is additively reduced, the inclusion $\supp(p)\subseteq\supp(pq)=\supp(f)$ holds. This implies that $\deg(p) \in \{g_1, \ldots, g_{n - 1}\}$. Similarly, $\deg(q) \in \{g_1, \ldots, g_{n - 1}\}$. Thus,
\[
    g_0 = \deg(f) = \deg(p) + \deg(q) \leq 2g_1,
\]
from which statement $(1)$ readily follows. In order to establish $(2)$, assume toward a contradiction that $g_0 = 2g_1$. In other words, the equality $\deg(p) = \deg(q) = g_1$ holds. Since $|\supp(f)|>3$, we may further assume that $p$ is not a binomial. Consequently, there exists $i \in \llbracket 2,n - 1 \rrbracket$ such that $g_1 + g_i \in \supp(f)$. In other words, the inequalities $g_1 < g_1 + g_i < g_1 + g_1 = g_0$ hold, a contradiction.

\vspace{-1.2\baselineskip}
\mbox{}\qedhere
\end{proof}

\begin{lemma}\label{lem33}
    Let \( S \) be an additively reduced semidomain and \( G \) a torsion-free (abelian) group. For an element $f=\sum_{i=0}^ns_ix^{g_i}\in S[G]$, the following statements hold.
    \smallskip
    \begin{itemize}
        \item[(1)] If $|\supp(f)|\ge2$ and $2g_{n-1}>g_0+g_n$, then $f$ is monolithic.
        \item[(2)] If $|\supp(f)|>3$ and $2g_{n-1}\ge g_0+g_n$, then $f$ is monolithic.
    \end{itemize}
\end{lemma}
\begin{proof}
We proceed as in the proof of Lemma~\ref{lem32}. Scaling by an appropriate power of $x$, we may assume that $g_n = 0$. Suppose that $f$ is not monolithic. Then we can write $f = pq$, where $p$ and $q$ are not monomial expressions in $S[G]$. Again, scaling by appropriate powers of $x$, we may assume that neither $p(0)$ nor $q(0)$ equals $0$. Since $S$ is additively reduced, the inclusion $\supp(p)\subseteq\supp(pq)=\supp(f)$ holds. Then the smallest positive element of $\supp(p)$ is greater than or equal to $g_{n - 1}$. Similarly, the smallest positive element of $\supp(q)$ is greater than or equal to $g_{n - 1}$. This implies that $2g_{n - 1} \leq g_0$, which concludes the proof of statement $(1)$. As for $(2)$, assume toward a contradiction that $g_0 = 2g_{n - 1}$. In other words, the equality $\deg(p) = \deg(q) = g_{n - 1}$ holds. Hence $f$ is a trinomial, which is a contradiction.
\end{proof}

The following two lemmas characterize the binomials and trinomials of a group semidomain \( S[G] \) (where $S$ is additively reduced, additively Furstenberg, and satisfies the condition \( \atom_+(S) = S^\times \)) which can be expressed as the sum of exactly two irreducible elements.

\begin{lemma}\label{lem34}
    Let $S$ be a semidomain that is additively reduced, additively Furstenberg, and satisfies $\atom_+(S)=S^\times$, and let $G$ be a torsion-free (abelian) group. Consider the element $f=s_0x^{g_0}+s_1x^{g_1} \in 
    S[G]$ with $s_0, s_1\in S^*$ and $g_0, g_1\in G$. The following statements hold.
    \smallskip
    \begin{itemize}
        \item[(1)] If either $s_0 \in S^{\times}$ or $s_1 \in S^{\times}$, then $f$ is irreducible.
        \item[(2)] If $s_0 \notin S^{\times}$ and $s_1 \notin S^{\times}$, then $f$ is the sum of two irreducible elements.
    \end{itemize}
\end{lemma}
\begin{proof}
Note that statement $(1)$ follows from Lemma~\ref{lem31}, along with the fact that binomials are always monolithic. As for $(2)$, since $S$ is additively Furstenberg and $\atom_+(S)=S^\times$, there exist elements $u_0, u_1 \in S^\times$ and $v_0, v_1 \in S^*$ such that $s_0 = u_0 + v_0$ and $s_1 = u_1 + v_1$. Thus,
\[
    f = [u_0x^{g_0} + v_1x^{g_1}] + [v_0x^{g_0} + u_1x^{g_1}],
\]
where both binomials between brackets are irreducible, which concludes our proof.

\mbox{}\qedhere
\end{proof}

\begin{lemma}\label{lem35}
   Let $S$ be a semidomain that is additively reduced, additively Furstenberg, and satisfies $\atom_+(S)=S^\times$, and let $G$ be a torsion-free (abelian) group. Consider the element $f=s_0x^{g_0}+s_1x^{g_1} + s_2x^{g_2} \in 
    S[G]$ with $s_0, s_1, s_2\in S^*$ and $g_0, g_1, g_2\in G$. Then $f$ is the sum of two irreducibles unless $f$ is the sum of three units, in which case it is irreducible.
\end{lemma}
\begin{proof}
If $s_1 \in S^{\times}$, then we claim that $f$ is irreducible. Suppose toward a contradiction that $f = pq$ for non-units $p,q \in S[G]$. Since $f$ is a trinomial and $s_1 \in S^{\times}$, both $p$ and $q$ are binomials. In this case, it is not hard to see that $s_1 = s' + s''$ for some $s', s'' \in S^*$, contradicting that $s_1 \in \atom_+(S)$. Hence, our claim follows. In particular, if $f$ is the sum of three units, then $f$ is irreducible.

For the rest of the proof, assume that $f$ is not the sum of three units. For each $i \in \llbracket 0,2 \rrbracket$, we can write $s_i = u_i + v_i$ with $u_i \in S^{\times}$ and $v_i \in S$ since $S$ is additively Furstenberg and $\atom_+(S)=S^\times$. We split our reasoning into the following two cases.

\noindent\textsc{Case $1$:} $s_1 \in S^{\times}$. In this case, either $v_0 \neq 0$ or $v_2 \neq 0$. Without loss of generality, suppose that $v_0 \neq 0$. We can then write
\[
    f = [v_0x^{g_0} + s_1x^{g_1} + v_2x^{g_2}] + [u_0x^{g_0} + u_2x^{g_2}],
\]
where the second summand between brackets is clearly irreducible and the first summand between brackets is either a trinomial (thus irreducible by the previous observation) or a binomial (thus irreducible as $s_1 \in S^{\times}$).

\noindent\textsc{Case $2$:} $s_1 \notin S^{\times}$. In this case, we can write
\[
    f = [v_0x^{g_0} + u_1x^{g_1} + s_2x^{g_2}] + [u_0x^{g_0} + v_1x^{g_1}]
\]
and, reasoning as above, it is not hard to show that both summands between brackets are irreducible.
\end{proof}

Now we are in a position to prove the main result of this section.

\begin{thm}\label{thm36}
    Let \( S \) be an additively reduced and additively Furstenberg semidomain and \( G \) a torsion-free (abelian) group. The following statements are equivalent.
    \begin{itemize}
        \item[(1)] $\atom_+(S)=S^\times$.
        \item[(2)] Every $f\in S[G]$ with $|\supp(f)|>1$ can be expressed as the sum of at most two irreducibles.
        \item[(3)] Every $f\in S[G]$ with $|\supp(f)|>1$ can be expressed as the sum of at most $n$ irreducibles for some $n \in \mathbb
        {N}_{>1}$.
    \end{itemize}
    Moreover, if any of the previous statements hold and $f\in S[G]$ is not of one of the following forms:
    \begin{itemize}
        \item[(a)] $f=s_0x^{g_0}+s_1x^{g_1}$, where either $s_0\in S^\times$ or $s_1\in S^\times$, or
        \item[(b)] $f=s_0x^{g_0}+s_1x^{g_1}+s_2x^{g_2}$, where either $s_0,s_1,s_2\in S^\times$,
    \end{itemize}
    then the polynomial expression $f$ can be written as the sum of exactly two irreducibles.
\end{thm}
\begin{proof}
It is obvious that $(2)$ implies $(3)$. Next we focus on proving that $(3)$ implies $(1)$. Since $S$ is additively reduced and additively Furstenberg, we know that $\atom_{+}(S) \neq \emptyset$. This implies that $S^{\times} \subseteq \atom_{+}(S)$. Indeed, if $u \in S^{\times}$ and $u = s' + s''$ for some $s', s'' \in S^*$, then every element $s \in S^*$ could be written as the sum of two nonzero elements, given by $$s = su^{-1}u = su^{-1}s' + su^{-1}s''.$$ In other words, the set of additive atoms $\atom_+(S)$ would be empty, which is impossible. Now suppose toward a contradiction that there exists an element $s\in\atom_+(S)\setminus S^\times$. Observe that, for any $n\in\NN$, the polynomial expression $\sum_{i=0}^{n}sx^{g_i} \in S[G]$ cannot be written as the sum of $n$ irreducibles as each term $sx^{g_i}$ cannot be decomposed additively. This contradiction proves that our hypothesis is untenable. Hence $\atom_+(S) = S^\times$.

It remains to verify that $(1)$ implies $(2)$. Let $f=\sum_{i=0}^ns_ix^{g_i}\in S[G]$, where $s_i \in S^*$ and $g_i \in G$ for every $i \in\llbracket 0,n \rrbracket$. If $f$ is either a binomial or a trinomial, then our result follows from Lemma~\ref{lem34} or Lemma~\ref{lem35}, respectively. From now on we assume that $|\supp(f)|>3$. Scaling $f$ by an appropriate power of $x$, we may further assume that $g_n = 0$. Since the semidomain $S$ is additively Furstenberg, we can write $s_0=u_0+v_0$ and $s_n=u_n+v_n$ for some $u_0, u_n \in S^\times$ and $v_0,v_n\in S$. Now let us define the polynomial expression $f'=\sum_{i=0}^ns'_ix^{g_i}$, where 
\[
    s'_i=
         \begin{cases}
            v_0 & \text{if } i=0,\\
            s_i & \text{if } 2g_i>g_0,\\
            u_n & \text{if } i=n,\\
            0   & \text{otherwise.}
         \end{cases}
\]
By Lemma~\ref{lem33}, the polynomial expression $f'$ is monolithic, and since $s_n' \in S^{\times}$, we have that either $f'$ is a unit or $f'$ is an irreducible by virtue of Lemma~\ref{lem31}. In a similar manner, we define the polynomial expression $f''=\sum_{i=0}^ns''_ix^{g_i}$, where 
\[
    s''_i=
          \begin{cases}
            u_0 & \text{if } i=0,\\
            s_i & \text{if } 2g_i<g_0,\\
            v_n & \text{if } i=n,\\
            0 & \text{otherwise.}
          \end{cases}
\]
Again, by Lemma~\ref{lem32}, the polynomial expression $f''$ is monolithic, and since $s_0'' \in S^{\times}$, we have that either $f''$ is a unit or $f''$ is an irreducible by virtue of Lemma~\ref{lem31}. Since $|\supp(f)|>3$, we know that one of $f'$ and $f''$ is not a unit. We have two possible cases.

\noindent\textsc{Case $1$:} Either $f'$ or $f''$ is a unit. Without loss of generality, assume that $f'$ is a unit. In this case, we have $v_0 = 0$ and $2g_i\le g_0$ for all $i \in \llbracket 1,n - 1 \rrbracket$. We can then write 
\begin{equation*}
    f = \left[f' + s_1x^{g_1}\right] + \left[f'' - s_1x^{g_1}\right],
\end{equation*}
where the summands between brackets are clearly irreducible.

\noindent\textsc{Case $2$:} Neither $f'$ nor $f''$ is a unit. In this case, both $f'$ and $f''$ are irreducibles in $S[G]$. It is not hard to see that $f = f'+f''$ as long as there does not exist an index $j \in \llbracket 1,n - 1 \rrbracket$ such that $2g_j = g_0$. Suppose then that such an index $j \in \llbracket 1,n - 1 \rrbracket$ exists. If $|\supp(f')| \ge 3$ (resp., $|\supp(f'')| \ge 3$), then $f' + s_jx^{g_j}$ (resp., $f'' + s_jx^{g_j}$) is irreducible by Lemma~\ref{lem31} and Lemma~\ref{lem33} (resp., Lemma~\ref{lem32}), from which our argument concludes. So, let us assume that $|\supp(f')|=|\supp(f'')| = 2$. Note that if $v_0 \neq 0$ and $v_n \neq 0$, then $|\supp(f)| = 3$, which is impossible. On the other hand, if $v_0 = 0$ (resp., $v_n = 0$), then $f' + s_jx^{g_j}$ (resp., $f'' + s_jx^{g_j}$) is irreducible by Lemma~\ref{lem31} and Lemma~\ref{lem33} (resp., Lemma~\ref{lem32}).

It is easy to see that the second part of the theorem follows from Lemma~\ref{lem34} and Lemma~\ref{lem35}.
\mbox{}\qedhere
\end{proof}

In~\cite{kaplan2023goldbach}, the authors studied analogues of the Goldbach conjecture for polynomials with coefficients in an additively reduced and additively atomic semidomain. In Theorem~\ref{thm36}, we considered polynomial expressions with coefficients in an additively reduced and additively Furstenberg semidomain. We conclude this section by presenting an additively Furstenberg semidomain that is not additively atomic.

\begin{ex}
    Let $p$ and $q$ be prime numbers satisfying $\frac{q}{p}>\frac{1+\sqrt{5}}2$, and consider the semidomain 
    \[
        S_{pq} =\NN_0\left[\frac{q}{p+q},\frac{q^2}{p^2+pq}\right].
    \]
    Note that our bound on $\frac{q}{p}$ guarantees that $\frac{q}{p+q}<1<\frac{q^2}{p^2+pq}$. We start by proving that $\atom_+(S_{pq}) = \{(\frac{q}{p+q})^n \mid n\in\NN_0\}$. Observe that $$C = \left\{\left(\frac{q}{p + q}\right)^n\left(\frac{q^2}{p^2 + pq}\right)^m \;\bigg|\; n,m \in \mathbb{N}_0\right\}$$ is a generating set of the additive monoid $(S_{pq},+)$. Consequently, $\atom_+(S_{pq})\subseteq C$. Since
    \begin{equation}\label{eq: identity}
        \frac{q^2}{p^2+pq}=\left(\frac{q}{p+q}\right)\left(\frac{q^2}{p^2+pq}\right)+\left(\frac{q}{p+q}\right)^2\!,
    \end{equation}
     we can conclude that $\atom_+(S_{pq})\subseteq\{(\frac{q}{p+q})^n \mid n\in\NN_0\}$. Suppose toward a contradiction that $(\frac{q}{p + q})^k$ is not an additive atom for some $k \in \mathbb{N}_0$. Then we can write
    \begin{equation} \label{eq: finding atoms}
        \left(\frac{q}{p + q}\right)^k = \sum_{i = 0}^t c_i\left(\frac{q}{p + q}\right)^{n_i}\left(\frac{q^2}{p^2 + pq}\right)^{m_i}\!,
    \end{equation}
    where $t \in \mathbb{N}$, the coefficients $c_0, \ldots, c_t \in \mathbb{N}$, and the exponents $n_0, m_0, \ldots, n_t, m_t \in \mathbb{N}_0$. Since $p^2 + pq < q^2$, we have that $k \leq n_i$ for every $i \in \llbracket 0,t \rrbracket$. After clearing out denominators in Equation~\eqref{eq: finding atoms}, we obtain that either $q \mid p$ or $q \mid p + q$, a contradiction. Consequently, the semidomain $S_{pq}$ is not additively atomic. On the other hand, by raising the identity~\eqref{eq: identity} to sufficiently high powers, it is not hard to show that, for every $s \in S_{pq}^*$, there exists $t_s \in \mathbb{N}_0$ such that $(\frac{q}{p + q})^{t_s} \mid_{(S_{pq},+)} s$. In other words, the semidomain $S_{pq}$ is additively Furstenberg.   
\end{ex}

\smallskip

\section{A Weak Goldbach Theorem for Group Series Semidomains}
\smallskip
In this section, we present analogues of the weak Goldbach conjecture for group series semidomains \( S\sg \), where \( S \) is an additively reduced and additively Furstenberg semidomain and \( G \) is a torsion-free abelian group. To do so, we extend several key results from \cite[Section~4]{kaplan2023goldbach}.

As before, we start by establishing a relationship between the irreducible and monolithic concepts in the context of series expressions.

\begin{lemma}\label{lem41}
    Let \( S \) be an additively reduced semidomain and \( G \) a torsion-free (abelian) group. Then $f=\sum_{i=0}^\infty s_ix^{g_i} \in S\sg$ with $|\supp(f)|>1$ is irreducible if and only if $f$ is monolithic and $1\in\gcdd(\{s_i \mid i \in \mathbb{N}_0\})$.
\end{lemma}
\begin{proof}
Suppose that $f$ is irreducible. Now if $f = pq$ for some $p,q \in S\sg^*$, then either $p$ or $q$ is a multiplicative unit and thus a monomial. Hence $f$ is monolithic. On the other hand, take $t \in S^*$ to be a common divisor of $\{s_i \mid i \in \mathbb{N}_0\}$. Since $f$ is irreducible, we obtain that $t \in S^{\times}$, which concludes the proof of the direct implication. Conversely, suppose that $f$ is monolithic and $1\in\gcdd(\{s_i \mid i \in \mathbb{N}_0\})$. Assume we can write $f = pq$ for some $p,q \in S\sg$. Since $f$ is monolithic, we may assume that $p = sx^g$ for some $s \in S^*$ and $g \in G$. Note that $s \in S^{\times}$ because $1\in\gcdd(\{s_i \mid i \in \mathbb{N}_0\})$. Consequently, the factor $p$ is necessarily a unit which, in turn, implies that $f$ is irreducible.
\end{proof}

The next two criteria for irreducibility will be crucial in proving the main result of this section.

\begin{lemma}\label{lem42}
    Let $S$ be an additively reduced semidomain and \( G \) a torsion-free (abelian) group. If $f=\sum_{i=0}^\infty s_ix^{g_i}\in S\sg^*$ satisfies that $g_1-g_0<g_{i+1}-g_{i}$ for every $i \in \mathbb{N}$, then $f$ is monolithic.
\end{lemma}
\begin{proof}
If $|\supp(f)| < 3$, then our result follows trivially. For the rest of the proof, we assume that none of $s_0$, $s_1$, and $s_2$ equal $0$. By way of contradiction, suppose that $f$ is not monolithic. Then there exist $p,q\in S\sg$ such that $f=pq$ and neither $p$ nor $q$ is a monomial. Write $p=\sum_{i=0}^\infty a_ix^{r_i}$ and $q=\sum_{i=0}^\infty b_ix^{t_i}$ with $a_0,a_1,b_0,b_1\in S^*$. Assume without loss of generality that $r_1-r_0\ge t_1-t_0$. Clearly, the equality $g_0=r_0+t_0$ holds. Since $S$ is additively reduced and $r_0 + t_1 \leq r_1 + t_0$, we have that $g_1 = r_0 + t_1$. Observe that $\{r_1+t_0,r_1+t_1\}\subseteq\supp(f)$ and $t_0<t_1$, implying the existence of some $j\in\NN$ with $g_j\ge r_1+t_0$ and $g_{j+1}=r_1+t_1$. This implies that $g_{j+1}-g_j\le t_1-t_0=g_1-g_0$, contradicting our assumption that $g_1-g_0<g_{i+1}-g_{i}$ for all $i\in\NN$. Hence $f$ is monolithic.
\end{proof}

\begin{lemma}\label{lem43}
    Let $S$ be an additively reduced semidomain and \( G \) a torsion-free (abelian) group. Consider the series expression $f=\sum_{i=0}^\infty s_ix^{g_i}\in S\sg$ with $s_i\in S^*$ for all $i\in\NN_0$. If the sequence $(g_{N+i+1}-g_{N+i})_{i\in\NN_0}$ is strictly increasing and unbounded above for some $N\in\NN_0$, then $f$ is monolithic.
\end{lemma}
\begin{proof}
Suppose towards a contradiction that $f$ is not monolithic. Then there exist $p,q\in S\sg$ such that $f=pq$ and neither $p$ nor $q$ is a monomial. Write $p=\sum_{i=0}^\infty a_ix^{r_i}$ and $q=\sum_{i=0}^\infty b_ix^{t_i}$, where $a_0,a_1,b_0,b_1\in S^*$. Since $f$ is not a polynomial expression, we may assume that $p$ is not a polynomial expression (i.e., $a_i \in S^*$ for every $i \in \NN_0$). For some sufficiently large $M \in \NN_{>N}$, we have that $g_{M+1}-g_M>t_1-t_0$. Since $S$ is additively reduced, the inclusion $\{r_{M+1}+t_0,r_{M+1}+t_1\}\subseteq\supp(f)$ holds. As a consequence, there exists $j\in\NN_0$ such that $g_j=r_{M+1}+t_0$ and $g_{j+1}\le r_{M+1}+t_1$ which, in turn, implies that $g_{j+1}-g_j\le t_1-t_0<g_{M+1}-g_M$. Since the sequence $(g_{N+i+1}-g_{N+i})_{i\in\NN_0}$ is strictly increasing, we get $j<M$. However, observe that $g_j=r_{M+1}+t_0$ is greater than $r_i+t_0$ for all $i\in \llbracket 0,M \rrbracket$, which implies $j>M$, a contradiction. Therefore, the series expression $f$ is monolithic.
\end{proof}

The next lemma establishes a key property of torsion-free groups in which every bounded-below subset is well-ordered. This property is crucial for proving the main result of this section.

\begin{lemma}\label{prop44}
    Let \( G \) be a torsion-free (abelian) group in which every subset that is bounded below is well-ordered. Then, for any strictly increasing sequence \((g_i)_{i \in \mathbb{N}_0}\) in \( G \), there exists a sequence \((n_i)_{i \in \mathbb{N}_0}\) of positive integers such that \((g_{n_{i+1}} - g_{n_i})_{i \in \mathbb{N}_0}\) is also strictly increasing.
\end{lemma}

\begin{proof}
We start the construction of the sequence \((n_i)_{i \in \mathbb{N}_0}\) by setting $n_0 \coloneqq 0$ and $n_1 \coloneqq 1$. Suppose that, for some $k \in \NN$, we have defined $n_0, \ldots, n_k$ such that the finite sequence $(g_{n_{i + 1}} - g_{n_i})_{i \in \llbracket 0, k - 1 \rrbracket}$ is strictly increasing. Note that there must exist some $M \in \NN_{>n_k}$ such that $g_M - g_{n_k} > g_{n_k} - g_{n_{k - 1}}$. Indeed, if not, the set $\{2g_{n_k} - g_{n_{k - 1}} - g_{\ell} \mid \ell \in \NN_{>n_k}\}$, which is bounded below by $0$, would not contain a minimal element. Now set $n_{k + 1} \coloneqq M$. Clearly, the sequence $(g_{n_{i + 1}} - g_{n_i})_{i \in \llbracket 0, k \rrbracket}$ is strictly increasing. An inductive argument shows that there exists a sequence \((n_i)_{i \in \mathbb{N}_0}\) of positive integers such that \((g_{n_{i+1}} - g_{n_i})_{i \in \mathbb{N}_0}\) is strictly increasing.
\end{proof}

By combining Lemma~\ref{lem43} and Lemma~\ref{prop44}, we arrive at the following monolithicity criterion.

\begin{cor}\label{lem43special}
    Let $S$ be an additively reduced semidomain and \( G \) a torsion-free (abelian) group in which every subset that is bounded below is well-ordered. Consider the series expression $f=\sum_{i=0}^\infty s_ix^{g_i}\in S\sg$ with $s_i\in S^*$ for all $i\in\NN_0$. If the sequence $(g_{N+i+1}-g_{N+i})_{i\in\NN_0}$ is strictly increasing for some $N\in\NN_0$, then $f$ is monolithic.
\end{cor}
\begin{proof}
    It suffices to show that the sequence $(g_{N+i+1}-g_{N+i})_{i\in\NN_0}$ is unbounded above. Suppose by way of contradiction that there exists $g \in G$ such that $g_{N+i+1}-g_{N+i} \leq g$ for every $i \in \NN_0$. Since $(g_{N+i+1}-g_{N+i})_{i\in\NN_0}$ is increasing, it follows that $\{g-g_{N+i-1}+g_{N+i}\mid i\in\NN_0\}$ is bounded below but is not well-ordered, a contradiction. Our result follows from Lemma~\ref{lem43}.
\end{proof}

We are now ready to prove the main result of this section.

\begin{thm}\label{thm45} Let $S$ be an additively reduced semidomain and \( G \) a torsion-free (abelian) group in which every subset that is bounded below is well-ordered. The following statements are equivalent.
\smallskip
    \begin{itemize}
        \item[(1)] $\atom_+(S)=S^\times$.
        \item[(2)] Every $f\in S\sg$ with $|\supp(f)|>1$ can be expressed as the sum of at most three irreducibles.
        \item[(3)] Every $f\in S\sg$ with $|\supp(f)|>1$ can be expressed as the sum of at most $n$ irreducibles for some $n \in \NN_{>2}$.
    \end{itemize}
\end{thm}
\begin{proof}
For the first part of the argument, we proceed as in the proof of Theorem~\ref{thm36}: It is obvious that $(2)$ implies $(3)$. Next we focus on proving that $(3)$ implies $(1)$. Since $S$ is additively reduced and additively Furstenberg, we know that $\atom_{+}(S) \neq \emptyset$. This implies that $S^{\times} \subseteq \atom_{+}(S)$. Indeed, if $u \in S^{\times}$ and $u = s' + s''$ for some $s', s'' \in S^*$, then every element $s \in S^*$ could be written as the sum of two nonzero elements, given by $$s = su^{-1}u = su^{-1}s' + su^{-1}s''.$$ In other words, the set of additive atoms $\atom_+(S)$ would be empty, which is impossible. Now suppose toward a contradiction that there exists an element $s\in\atom_+(S)\setminus S^\times$. Observe that, for any $n\in\NN$, the series expression $\sum_{i=0}^{\infty}sx^{g_i} \in S\sg$ cannot be written as the sum of $n$ irreducibles as each term $sx^{g_i}$ cannot be decomposed additively. This contradiction proves that our hypothesis is untenable. Hence $\atom_+(S) = S^\times$. 

It remains to verify that $(1)$ implies $(2)$. Let $f=\sum_{i=0}^{\infty}s_ix^{g_i}$ with $s_i \in S$ and $g_i \in G$ for every $i \in \NN_0$. Since $S$ is additively reduced, $S[G]^*$ is a divisor-closed submonoid of $S\sg^*$. Consequently, the inclusion $\atom(S[G])\subseteq\atom(S\sg)$ holds. By Theorem~\ref{thm36}, we may assume that $f$ is not a polynomial expression (i.e., $s_i \in S^*$ for each $i \in \NN_0$). Since every bounded-below subset of $G$ is well-ordered, we can define  
\[
    \delta=\min\{g_{i+1}-g_i \mid i\in\NN_0\} \hspace{.5 cm} \text{ and } \hspace{.5 cm} \Delta=\{j\in\NN_0 \mid g_{j+1}-g_j=\delta\}.
\]
We split our reasoning into the following two cases. 

\noindent\textsc{Case 1:} $\Delta$ is finite. In this case, let $\mu=\min \Delta$ and $\nu=\max \Delta$. There exist $p,q\in S\sg$ such that $p=\sum_{i=0}^\infty a_ix^{r_i}$ and $q=\sum_{i=0}^\infty b_ix^{t_i}$ satisfy the following properties:
\begin{itemize}
    \item[(a)] $(a_0,a_1,r_0,r_1)=(s_\mu,s_{\mu+1},g_\mu,g_{\mu+1})$,
    \item[(b)] $r_1-r_0<r_{i+1}-r_i$ for every $i \in \NN$, 
    \item[(c)] $(t_{N+i+1}-t_{N+i})_{i\in\NN_0}$ is strictly increasing for some $N \in \NN_0$,
    \item[(d)] $f=p+q$, and
    \item[(e)] $b_i\in S^*$ for each $i \in \NN_0$.
\end{itemize}
In fact, one can construct such series expressions $p,q\in S\sg$ by letting $p$ and $q$ contain the polynomial expressions $p'=s_\mu x^{g_{\mu}}+s_{\mu+1}x^{g_{\mu+1}}$ and $q'=\sum_{i=0}^{\mu-1}s_ix^{g_i}+\sum_{i=\mu+2}^{\nu+1}s_ix^{g_i}$, respectively. It is not hard to see that $p$ satisfies conditions (a) and (b). Additionally, Lemma~\ref{prop44} guarantees that we can select specific exponents from $\supp(f)$ in our construction of $\supp(q)$ so that the sequence $(t_{i+1}-t_{i})_{i\in\NN_0}$ is eventually strictly increasing. Properties (d) and (e) can be maintained by appropriately distributing the terms of \( f \). Our assumptions that $\atom_+(S)=S^\times$ and that $S$ is additively Furstenberg imply the existence of $u_2,u_3\in S^\times$ and $v_2,v_3\in S$ such that $b_{M + 2}=u_2+v_2$ and $b_{M + 3}=u_3+v_3$ for some $M \in \NN_{> \nu}$. Hence setting $p''\coloneqq p+v_2x^{t_{M + 2}}+u_3x^{t_{M + 3}}$ and $q''\coloneqq q-v_2x^{t_{M + 2}}-u_3x^{t_{M + 3}}$ yields that $f = p''+q''$.
By Lemma~\ref{lem41}, Lemma~\ref{lem42}, and Corollary~\ref{lem43special}, one can check that both $p''$ and $q''$ are irreducibles.

\noindent\textsc{Case 2:} $\Delta$ is infinite. Utilizing a method similar to that of the previous case, we can construct series expressions $p,q\in S\sg$ with $p=\sum_{i=0}^\infty a_ix^{r_i}$ and $q=\sum_{i=0}^\infty b_ix^{t_i}$ satisfying the following properties:
\begin{itemize}
    \item[(a)] $r_1-r_0=r_3-r_2= \delta$,
    \item[(b)] $a_0,a_2,b_j\in S^\times$ for some $j$ sufficiently large, 
    \item[(c)] $(t_{N+i+1}-t_{N+i})_{i\in\NN_0}$ strictly increasing for some $N \in \NN$,
    \item[(d)] $f=p+q$, and
    \item[(e)] $b_i\in S^*$ for each $i \in \NN_0$.
\end{itemize}
Now set $p'\coloneqq a_0x^{r_0}+a_1x^{r_1}+\sum_{i=0}^\infty a_{2i+4}x^{r_{2i+4}}$ and $p''\coloneqq a_2x^{r_2}+a_3x^{r_3}+\sum_{i=0}^\infty a_{2i+5}x^{r_{2i+5}}$. We then note that $f = p' + p'' + q$. One can again check that $p'$, $p''$, and $q$ are irreducibles by Lemma~\ref{lem41}, Lemma~\ref{lem42}, and Corollary~\ref{lem43special}.
\end{proof}

Using Theorem~\ref{thm45}, we can now provide an analogue of the weak Goldbach conjecture for group series semidomains $S\sg$, where $S$ is additively reduced and additively Furstenberg and $G$ is a finitely generated torsion-free group. 

\begin{cor}\label{cor: Goldbach for finitely generated groups} Let $S$ be an additively reduced semidomain and \( G \) a finitely generated torsion-free (abelian) group. The following statements are equivalent.
\smallskip
    \begin{itemize}
        \item[(1)] $\atom_+(S)=S^\times$.
        \item[(2)] Every $f\in S\sg$ with $|\supp(f)|>1$ can be expressed as the sum of at most three irreducibles.
        \item[(3)] Every $f\in S\sg$ with $|\supp(f)|>1$ can be expressed as the sum of at most $n$ irreducibles for some $n \in \NN_{>2}$.
    \end{itemize}
\end{cor}
\begin{proof}
    It is obvious that (2) implies (3). Moreover, note that our proof that (3) implies (1) in Theorem~\ref{thm45} does not depend on the group $G$. Consequently, we can focus on showing that (1) implies (2). 
    
    By the Fundamental Theorem of Finitely Generated Abelian Groups, we know that $G$ is free abelian and thus isomorphic to $\ZZ^n$ for some $n \in \NN$. We proceed by induction on the rank $n$. For $n = 1$, our result follows from Theorem~\ref{thm45} as every bounded-below subset of $\ZZ$ is well-ordered. Suppose that every $f\in S\llbracket \ZZ^n \rrbracket$ with $|\supp(f)|>1$ can be expressed as the sum of at most three irreducibles. Observe that $S\llbracket \ZZ^n \rrbracket$ is additively reduced given that $S$ is additively reduced. Moreover, it is not hard to see that $\atom_+(S\llbracket \ZZ^n \rrbracket)=S\llbracket \ZZ^n \rrbracket^\times$. Now since $S\llbracket \ZZ^{n + 1} \rrbracket$ is isomorphic to $S\llbracket \ZZ^{n} \rrbracket\llbracket \ZZ \rrbracket$ and every subset of $\ZZ$ that is bounded below is well-ordered, every $f\in S\llbracket \ZZ^{n + 1} \rrbracket$ with $|\supp(f)|>1$ can be expressed as the sum of at most three irreducibles by Theorem~\ref{thm45}.  
\end{proof}

Note that Corollary~\ref{cor: Goldbach for finitely generated groups} addresses cases not covered by Theorem~\ref{thm45}. For instance, with respect to the lexicographical order, the subset \( \{(0, -m) \mid m \in \mathbb{N}_0\} \) of \( \mathbb{Z}^2 \) is bounded below but not well-ordered, as it contains no minimum element.

\smallskip
\section*{Acknowledgments}
The authors express gratitude to their mentor, Dr. Harold Polo, for his invaluable guidance, assistance, and support throughout the research and writing phases. They would also like to thank Dr. Felix Gotti for his indispensable feedback and suggestions on the paper. Finally, the authors extend their appreciation to the MIT PRIMES-USA program for enabling this research opportunity. 


\smallskip
\begin{thebibliography}{20}
	
	\bibitem{bender} A. O. Bender: \emph{Representing an element in $\mathbb{F}_q[t]$ as the sum of two irreducibles}, Mathematika \textbf{60} (2014) 166--182.
	
	\bibitem{CF19} F. Campanini and A. Facchini: \emph{Factorizations of polynomials with integral non-negative coefficients}, Semigroup Forum \textbf{99} (2019) 317--332.
	
	\bibitem{car} M. Car and L. H. Gallardo: \emph{Representation of a polynomial as the sum of an irreducible polynomial and a square-free polynomial}, Acta Arith. \textbf{197} (2021) 293--309.
	
	\bibitem{chapmanpolo} S. T. Chapman and H. Polo: \emph{Arithmetic of additively reduced monoid semidomains}, Semigroup Forum \textbf{107} (2023) 40--59.

    \bibitem{effinger} G. W. Effinger and D. R. Hayes: \emph{A complete solution to the polynomial $3$-primes problem}, Bull. Amer. Math. Soc. \textbf{24} (1991) 363--369.
	
    \bibitem{fox2023} H. Fox, A. Goel, and S. Liao: \emph{Arithmetic of semisubtractive semidomains}, J. Algebra Appl. (Online Ready) https://doi.org/10.1142/S0219498825501634. 

	\bibitem{golan} J. S. Golan: \emph{Semirings and their Applications}, Kluwer Academic Publishers, 1999. 
	
	\bibitem{polo} F. Gotti and H. Polo: \emph{On the arithmetic of polynomial semidomains}, Forum Math. (2023) \textbf{35} 1179--1197.  
	
	\bibitem{hayes} D. R. Hayes: \emph{A Goldbach theorem for polynomials with integral coefficients}, Amer. Math. Monthly \textbf{72} (1965) 45--46.

     \bibitem{kaplan2023goldbach} N. Kaplan and H. Polo: \emph{A Goldbach theorem for Laurent series semidomains}, arXiv:https://arxiv.org/abs/2312.14888.
	
	\bibitem{kozek} M. Kozek: \emph{An asymptotic formula for Goldbach's conjecture with monic polynomials in $\mathbb{Z}[x]$}, Amer. Math. Monthly \textbf{117} (2010) 365--369.
	
    \bibitem{levi1913} F. Levi: \textit{Arithmetische gesetze im gebiete diskreter gruppen}, Rend. Circ. Matem. Palermo \textbf{35} (1913) 225--236.
 
	\bibitem{liaopolo} S. Liao and H. Polo: \emph{A Goldbach theorem for Laurent polynomials with positive integer coefficients}, Amer. Math. Monthly \textbf{131} (2024) 704--711.

	\bibitem{paran} E. Paran: \emph{Twin-prime and Goldbach theorems for $\mathbb{Z}\llbracket x\rrbracket$}, J. Number Theory \textbf{213} (2020) 453--461.
		
	\bibitem{pollack} P. Pollack: \emph{On polynomial rings with a Goldbach property}, Amer. Math. Monthly \textbf{118} (2011) 71--77.
    
	\bibitem{saidak} F. Saidak: \emph{On Goldbach's conjecture for integer polynomials}, Amer. Math. Monthly \textbf{113} (2006) 541--545.

\end{thebibliography}
\end{document}